\documentclass[a4paper,reqno]{amsart}

\usepackage{amssymb,amsmath}
\usepackage{fullpage}
\usepackage{paralist,enumitem}
\setlist{listparindent=0pt,parsep=3pt}
\setdefaultenum{1.}{}{}{}

\usepackage[pdftex,linktocpage,pdfstartview=FitH,colorlinks=true,linkcolor=blue,citecolor=magenta]{hyperref}
\usepackage{amsthm} 
\usepackage[capitalise]{cleveref}

\usepackage[initials,msc-links]{amsrefs}[2007/10/22]

\newtheorem{thm}{Theorem}[section]
\newtheorem{prop}[thm]{Proposition}

\newtheorem*{proposition*}{Proposition}
\newtheorem*{theorem*}{Theorem}
\theoremstyle{definition}

\theoremstyle{remark}
\newtheorem*{rem}{Remark}

\numberwithin{equation}{section}

\newcommand{\R}{\mathbb{R}}
\newcommand{\C}{\mathbb{C}}
\newcommand{\Z}{\mathbb{Z}}

\newcommand{\N}{\mathbb{N}}

\newcommand{\I}{\mathrm{I}}

\DeclareSymbolFont{script}{U}{eus}{m}{n}
\DeclareSymbolFontAlphabet{\mathscr}{script}
\DeclareMathSymbol{\EuWedge}{0}{script}{"5E}
\newcommand{\Wedge}{\EuWedge}

\newcommand{\fso}{\mathfrak{so}}

\newcommand{\fsu}{\mathfrak{su}}

\newcommand{\g}{\mathfrak{g}}

\renewcommand{\q}{\mathfrak{q}}

\newcommand{\gc}{\g^{\C}}

\newcommand{\set}[1]{\{#1\}}

\newcommand{\abs}[1]{\lvert#1\rvert}

\newcommand{\half}{\tfrac12}

\renewcommand{\I}{\sqrt{-1}}
\DeclareMathOperator{\ad}{ad}
\newcommand{\n}{\mathfrak{n}}

\title{On canonical elements}
\author{F.E. Burstall}

\begin{document}
\begin{abstract}
  We characterise the canonical elements, in the sense of
  Burstall--Rawnsley \cite{BurRaw90}, of a compact semisimple
  Lie algebra and discuss the case of $\fso(n)$ in detail.
  In so doing, we correct two errors in Burstall et
  al.\ \cite{BurEscFerTri04}. 
\end{abstract}
\maketitle

\section{Introduction}
\label{sec:introduction}

Let $\g$ be a compact semisimple Lie algebra with
complexification $\gc$.  Burstall--Rawnsley \cite{BurRaw90}
associate a unique element $\xi\in\g$ to each parabolic subalgebra
$\q\leq\gc$ with the property that the eigenspaces of $\ad\xi$
grade the central descending series of the nilradical of
$\q$.  They called $\xi$ the \emph{canonical element} of
$\q$.

Canonical elements provide a natural realisation of flag
manifolds (conjugacy classes of parabolic subalgebras) as
adjoint orbits in $\g$ and give rise to natural fibrations
of flag manifolds over Riemannian symmetric spaces that
factor through twistor space \cite{BurRaw90}.  Moreover,
canonical elements label certain Schubert cells
in the based loop group and, in this context, have been
a useful tool for studying harmonic maps of finite uniton
number \cite{BurGue97,CorPac15,FerSimWoo17}.

In the appendix to \cite{BurEscFerTri04}, Burstall et al.\
claim a characterisation of those elements of $\g$ that are
canonical for some parabolic subalgebra.  They also offer a
classification of the canonical elements of $\fsu(n)$ and
$\fso(n)$.  Unfortunately, their characterisation does not
capture all canonical elements and their classification,
while correct for $\fsu(n)$ and $\fso(2n+1)$, misses a
necessary condition in the $\fso(2n)$ case.

It is the purpose of this short note to remedy these
deficiencies.

\subsection*{Acknowledgements}
It is a pleasure to thank Joe Oliver and Martin Svensson for
pointing out the mistakes in \cite{BurEscFerTri04}, John
C. Wood for urging me to write this note, and David
Calderbank, for a very helpful conversation during its
preparation.

\section{Parabolic subalgebras and their canonical elements}
\label{sec:parab-subalg-their}

Let $\g$ be a compact semisimple Lie algebra with
complexification $\gc$.  Recall that a subalgebra
$\q\leq\gc$ is \emph{parabolic} if it contains a maximal
solvable subalgebra of $\gc$ or, equivalently, if the polar
$\q^{\perp}$ of $\q$ with respect to the Killing form is a
nilpotent subalgebra \cite[Lemma~4.2]{BurRaw90}.  Then
$\q^{\perp}$ is the nilradical of $\q$.

\begin{rem}
  It follows that $\gc$ itself is parabolic.  More
  generally, if $\q\leq\gc$ is parabolic and $\hat{\g}^{\C}$
  is another complex semisimple algebra then
  $\q\oplus\hat{\g}^{\C}$ is a parabolic subalgebra of
  $\gc\oplus\hat{\g}^{\C}$.  
\end{rem}

We can construct parabolic subalgebras of $\gc$ from the
eigenspaces of $\ad\xi$, for $\xi\in\g$.  Since $\ad\xi$ is
skew for the (negative definite) Killing form, it is
semisimple with eigenvalues in $\I\R$.  For $r\in\R$, let
$\g_r\leq\gc$ denote the $\I r$-eigenspace of $\ad\xi$ when $\I r$
is an eigenvalue and $\set{0}$ otherwise.  Then
\begin{equation*}
  [\g_r,\g_s]\leq \g_{r+s}\qquad \g_r^{\perp}=\sum_{s+r\neq
    0}\g_s\qquad \overline{\g_{r}}=\g_{-r},
\end{equation*}
where the complex conjugation is across the compact real
form $\g$.  It now follows at once that $\sum_{r\geq0}g_{r}$
is parabolic with nilradical $\sum_{r>0}g_r$.

Many elements of $\g$ give rise to the same parabolic
subalgebra but, for any parabolic subalgebra, there is a
canonical choice available whose positive eigenspaces
grade the central descending series of the nilradical.
Indeed, Burstall--Rawnsley prove:
\begin{thm}[{\cite[Theorem~4.4]{BurRaw90}}]
  \label{th:1}
  Let $\q\leq\gc$ be a parabolic subalgebra with nilradical
  $\n=\q^{\perp}$.  Then there is a unique $\xi\in\g$ such
  that:
  \begin{compactenum}[(i)]
  \item the eigenvalues of $\ad\xi$ lie in $\I\Z$;
  \item For $r\in\Z$, let $\g_r\leq\gc$ denote the $\I
    r$-eigenspace of $\ad\xi$ and $\n^{(r)}$ the $r$-th step
    of the central descending series of $\n$: thus
    $\n^{(1)}=\n$ and $\n^{r+1}=[\n,\n^{(r)}]$, for $r>1$.
    Then
    \begin{equation*}
      \q=\sum_{i\geq 0}g_i\qquad n^{(r)}=\sum_{i\geq r}g_i.
    \end{equation*}
  \end{compactenum}
  We call $\xi$ the \emph{canonical element} of $\q$.
\end{thm}

In this note, we characterise those $\xi\in\g$ which are
canonical for some parabolic subalgebra of $\gc$.

\section{Characterisation of canonical elements}
\label{sec:char-canon-elem}

Throughout this section, we fix $\xi\in\g$ and let
$\g_r\leq\gc$ denote the $\I r$-eigenspace of $\ad\xi$.  Thus
\begin{equation*}
  \gc=\bigoplus\g_{r}.
\end{equation*}
We prove:
\begin{thm}
  \label{th:2}
  $\xi$ is the canonical element of some parabolic
  subalgebra of $\gc$ if and only if
  \begin{compactenum}[(1)]
  \item $\ad\xi$ has eigenvalues in $\I\Z$;
  \item $\gc$ is generated by $\g_1\oplus\g_0\oplus\g_{-1}$.
  \end{compactenum}
\end{thm}
\begin{rem}
  In the appendix to \cite{BurEscFerTri04}, Burstall et al.\
  propose an alternative definition of canonical element
  where the second condition of \cref{th:2} is replaced by
  the more stringent requirement that $\gc$ be generated by
  $\g_1\oplus\g_{-1}$.  Moreover, they assert (footnote 8,
  page~64) that this definition coincides with our present
  one.  However, the example of $\xi=0$ (so that $\q=\gc$)
  shows that this assertion is false.  Even if we restrict
  attention to proper parabolic subalgebras, we get
  counter-examples from the canonical elements of parabolic
  subalgebras of the form $\q\oplus\hat{\g}^{\C}$ discussed
  above.  In particular, maximal proper parabolic
  subalgebras of $\fso(4)$ fall into this category.
\end{rem}
\begin{proof}[Proof of \cref{th:2}]
  Let $\q\leq\gc$ be parabolic with nilradical $\n$ and
  canonical element $\xi$.  Then $\xi$ clearly satisfies
  condition (1).  For condition (2), we have
  \begin{equation*}
    \gc=\n\oplus\g_0\oplus\overline{\n}
  \end{equation*}
  so it suffices to show that $\g_1$ generates $\n$ or,
  equivalently, that, for $k\geq1$,
  $[\g_1,\g_k]=\g_{k+1}$. For this, note that
  $[\g_1,\g_k]\leq\g_{k+1}$ and also
  \begin{equation*}
    [\n,\n^{(k)}]=\n^{(k+1)},\qquad \n^{(k)}=\g_k\oplus\n^{(k+1)}.
  \end{equation*}
  It follows that
  \begin{equation*}
    [\g_1,\g_k]\equiv\g_{k+1}\mod\n^{(k+2)},
  \end{equation*}
  whence the result.

  For the converse, suppose that $\xi$ satisfies the
  conditions (1) and (2). Define $\g^k$, $k\geq 1$,
  inductively by $\g^1=\g_1$ and
  $\g^{k+1}=[\g_1,\g^k]\leq\g_{k+1}$.  Further, set
  $\g^{0}=\g_0$ and $\g^{-k}=\overline{\g^k}$, for
  $k\geq 1$.  The parabolic $\q$ defined by $\xi$ has
  nilradical $\bigoplus_{k\geq 1}\g_k$ so that $\xi$ will be
  canonical as soon as $\g_k=\g^k$, for $k\geq 1$.  The
  Jacobi identity and an induction gives us
  \begin{equation*}
    [\g_0,\g^r]\leq\g^r,\qquad [\g_{\pm1},g^r]\leq \g^{r\pm1},
  \end{equation*}
  for all $r\in\Z$, and then a further induction yields
  \begin{equation*}
    [\g^r,\g^s]\leq \g^{r+s},
  \end{equation*}
  for all $r,s\in\Z$.  Thus $\bigoplus_{r\in\Z}\g^r$ is a
  subalgebra of $\gc$ containing
  $\g_1\oplus\g_0\oplus\g_{-1}$ so that, by condition (2),
  $\bigoplus_{r\in\Z}\g^r=\gc$.  In particular, $\g^k=\g_k$
  for all $k$ and the theorem is proved.
\end{proof}
\begin{rem}
  See Theorem 4.4 of \cite{BurRaw90} and remark (i)
  following it for a proof of the forward implication using
  roots with respect to a maximal torus in $\g_0$.
\end{rem}

\section{Example}
\label{sec:examples}

We apply \cref{th:2} to describe all the canonical elements
in $\fso(n)$.  Another approach would be to use the explicit
description of canonical elements in terms of roots
\cite{BurRaw90}: for this, see \cite{CorPac15}.

An element $\xi\in\fso(n)$ is determined up to conjugacy by
its eigenvalues and their multiplicities.  We will see that
a multiplicity can intervene in the condition for $\xi$ to be
canonical.

\begin{prop}
  \label{th:3}
  Let $\xi\in\fso(n)$.  Then $\xi$ is canonical if and only
  if either:
  \begin{compactenum}[(a)]
  \item For some $k\in\N$, $\xi$ has integer eigenvalues $\pm\I j$,
    $0\leq j\leq k$, or
  \item For some $k\in\N$, $\xi$ has half-integer
    eigenvalues $\pm\I(j+\half)$, $0\leq j\leq k$ \emph{and}
    the eigenvalues $\pm\half\I$ have multiplicity at least
    $2$.
  \end{compactenum}
\end{prop}
\begin{proof}
  Set $\g=\fso(n)$ and $\xi\in\g$ be canonical with
  eigenvalues $\I\Delta\subset\I\R$.  Let
  $V_{\lambda}\leq(\R^n)^{\C}=\C^n$ be the
  $\I\lambda$-eigenspace of $\xi$ so that
  $\C^n=\bigoplus_{\lambda\in\Delta}V_\lambda$.  Since $\xi$
  is real and skew for the (complex bilinear extension of
  the) inner product $(\,,\,)$ on $\R^n$, we have
  \begin{equation*}
    V^{\perp}_\lambda=\sum_{\lambda+\mu\neq 0}V_\mu,
    \qquad V_{-\lambda}=\overline{V_{\lambda}}.
  \end{equation*}
  In particular, we see that $\Delta=-\Delta$ and that
  $V^{*}_{\lambda}\cong V_{-\lambda}$, for all
  $\lambda\in\Delta$.

  Now, $\g_0V_{\lambda}\leq V_{\lambda}$, for all
  $\lambda\in\Delta$, while,
  \begin{equation*}
    \g_{\pm1}V_{\lambda}\leq 
    \begin{cases}
      V_{\lambda\pm1}&\text{if $\lambda\pm
        1\in\Delta$,}\\
      \set{0}&\text{otherwise.}
    \end{cases}
  \end{equation*}
  Let $\Delta=\set{\lambda_1<\dots<\lambda_{N}}$ and set
  \begin{equation*}
    W=\bigoplus_{\lambda\in\Delta, \lambda-\lambda_1\in\Z}V_{\lambda}.
  \end{equation*}
  It follows that $W$ is preserved by $\g_{\pm1}$ and $\g_0$
  and so, thanks to \cref{th:2}, by $\g$.  Since $\C^n$ is
  an irreducible $\g$-module, we conclude that $W=\C^n$ so
  that any two $\lambda_i,\lambda_j$ differ by an integer.
  We can say more: if some $\lambda_{i+1}-\lambda_i\geq2$
  then $\bigoplus_{j\leq i}V_{\lambda_j}$ is again
  $\g$-invariant which is a contradiction.  Thus successive
  $\lambda_i$ differ by $1$ which together with
  $\Delta=-\Delta$ forces $\Delta$ to have one of the two
  forms required depending on the parity of $\abs{\Delta}$.

  Finally, if $\pm\half$ has multiplicity $1$ with
  eigenvectors $v_{\pm}$ and $\eta\in\g_{1}$, then, since
  $\eta$ is skew, $0=(\eta v_{-},v_-)$ which forces
  $\eta v_-=0$ since $(-,v_{-})$ spans $V_{\frac12}^{*}$.
  Thus, in this case, $\g_{1}V_{-\frac12}=\set{0}$ and so
  $\bigoplus_{\lambda<0}V_{\lambda}$ is $\g$-invariant:
  again a contradiction.

  For the converse, let $\xi\in\fso(n)$ have eigenvalues
  $\I\lambda$ as in (a) or (b) with eigenspaces
  $V_{\lambda}\leq\C^n$: thus
  $-\lambda_0\leq\lambda\leq\lambda_0$ is an integer or
  half-integer, $\lambda_0=k$ or $k+\half$ and
  $\dim V_{\frac12}\geq 2$ in the latter case.  We use the
  $\gc$-equivariant identification
  $\Wedge^{2}\C^{n}\cong\gc$ given by
  \begin{equation*}
    (a\wedge b)(c)=(a,b)b-(b,c)a
  \end{equation*}
  so that
  \begin{equation*}
    \gc=\sum_{\lambda\leq\mu}V_{\lambda}\wedge V_{\mu}
  \end{equation*}
  with $\ad\xi$ having eigenvalue $\I (\lambda+\mu)$ on
  $V_{\lambda}\wedge V_{\mu}$.

  We therefore have
  \begin{equation*}
    \g_1=\bigoplus_{-\lambda_0<\lambda\leq\frac12}V_{\lambda}\wedge V_{1-\lambda}
  \end{equation*}
  with the multiplicity condition in case (b) ensuring that
  $\Wedge V_{\frac12}$ is non-zero.  Now let $\ell>1$ and
  $V_{\lambda}\wedge V_{\ell-\lambda}$ be a non-zero summand
  of $\g_{\ell}$ with $\lambda\leq\ell-\lambda$.  Then
  \begin{equation*}
    -\lambda_0\leq \lambda\leq\ell-\lambda\leq\lambda_0
  \end{equation*}
  so that, first, $-\lambda_0\leq\lambda-1<\ell-\lambda$
  ensuring that $V_{\lambda-1}\wedge V_{\ell-\lambda}$ is
  non-zero\footnote{The reason for circumspection here is
    that $V_{\lambda}\wedge V_{\lambda}$ could vanish when
    $\abs{\lambda}>\half$.}.  Second, we see that
  $-\lambda_0\leq\lambda,1-\lambda\leq\lambda_0$ so that
  $V_{\lambda}\wedge V_{1-\lambda}$ is non-zero also.  A
  direct calculation now shows that
  \begin{equation*}
    [V_{\lambda-1}\wedge V_{\ell-\lambda},V_{\lambda}\wedge
    V_{1-\lambda}]=V_{\lambda}\wedge V_{\ell-\lambda}.
  \end{equation*}
  It follows that $\g_{\ell}=[\g_1,\g_{\ell-1}]$, for all
  $\ell>1$, and we conclude that $\xi$ is canonical.
\end{proof}
\begin{rem}
  \cref{th:3} corrects
  \cite[Proposition~A.2]{BurEscFerTri04} which omits the
  condition on $\dim V_{\frac12}$ in case (b).
\end{rem}


\begin{bibdiv}
\begin{biblist}

\bib{BurEscFerTri04}{article}{
      author={Burstall, F.~E.},
      author={Eschenburg, J.-H.},
      author={Ferreira, M.~J.},
      author={Tribuzy, R.},
       title={K\"{a}hler submanifolds with parallel pluri-mean curvature},
        date={2004},
        ISSN={0926-2245},
     journal={Differential Geom. Appl.},
      volume={20},
      number={1},
       pages={47\ndash 66},
  url={https://doi-org.ezproxy1.bath.ac.uk/10.1016/S0926-2245(03)00055-X},
      review={\MR{2030166}},
}

\bib{BurGue97}{article}{
      author={Burstall, F.~E.},
      author={Guest, M.~A.},
       title={Harmonic two-spheres in compact symmetric spaces, revisited},
        date={1997},
        ISSN={0025-5831},
     journal={Math. Ann.},
      volume={309},
      number={4},
       pages={541\ndash 572},
         url={http://dx.doi.org/10.1007/s002080050127},
      review={\MR{MR1483823 (99f:58046)}},
}

\bib{BurRaw90}{book}{
      author={Burstall, Francis~E.},
      author={Rawnsley, John~H.},
       title={Twistor theory for {R}iemannian symmetric spaces},
      series={Lecture Notes in Mathematics},
   publisher={Springer-Verlag},
     address={Berlin},
        date={1990},
      volume={1424},
        ISBN={3-540-52602-1},
        note={With applications to harmonic maps of Riemann surfaces},
      review={\MR{MR1059054 (91m:58039)}},
}

\bib{CorPac15}{article}{
      author={Correia, Nuno},
      author={Pacheco, Rui},
       title={Harmonic maps of finite uniton number and their canonical
  elements},
        date={2015},
        ISSN={0232-704X},
     journal={Ann. Global Anal. Geom.},
      volume={47},
      number={4},
       pages={335\ndash 358},
         url={https://doi.org/10.1007/s10455-014-9448-7},
      review={\MR{3331893}},
}

\bib{FerSimWoo17}{article}{
      author={Ferreira, Maria~Jo\~{a}o},
      author={Sim\~{o}es, Bruno~Ascenso},
      author={Wood, John~C.},
       title={{Harmonic maps into the orthogonal group and
           null curves}},
       journal={Math. Z.}
        date={to appear},
}

\end{biblist}
\end{bibdiv}

\end{document}